\documentclass[11pt]{article}

\pdfoutput=1
\usepackage[backref]{hyperref}
\usepackage{amsmath,amsfonts,amssymb,latexsym,amsthm,dsfont}
\usepackage[a4paper,margin=2.5cm]{geometry}
\usepackage{graphicx}
\usepackage[utf8]{inputenc}
\usepackage[T1]{fontenc}
\usepackage{lmodern}
\usepackage{tikz,enumitem}
\usetikzlibrary{arrows,calc,decorations.text}
\usepackage[english]{babel}

\newtheorem{thm}{Theorem}[section]%
\newtheorem{cor}[thm]{Corollary}%
\newtheorem{prop}[thm]{Proposition}%
\newtheorem{lem}[thm]{Lemma}%

\newtheorem{xpl}[thm]{Example}%
\newtheorem{rem}[thm]{Remark}%
\newcommand{\dC}{\mathbb{C}}
\newcommand{\dE}{\mathbb{E}}

\newcommand{\dN}{\mathbb{N}}
\newcommand{\dP}{\mathbb{P}}
\newcommand{\dR}{\mathbb{R}}
\newcommand{\dS}{\mathbb{S}}

\newcommand{\cA}{\mathcal{A}}
\newcommand{\cC}{\mathcal{C}}
\newcommand{\cF}{\mathcal{F}}

\newcommand{\cM}{\mathcal{M}}

\let\abs\ABS
\newcommand{\BRA}[1]{{{\left\{#1\right\}}}} 
\newcommand{\DP}[1]{{{\left<#1\right>}}} 
\newcommand{\NRM}[1]{{{\left\| #1\right\|}}} 
\newcommand{\PAR}[1]{{{\left(#1\right)}}} 
\newcommand{\SBRA}[1]{{{\left[#1\right]}}} 
\renewcommand{\leq}{\leqslant}
\renewcommand{\geq}{\geqslant}

\newcommand{\ind}{\mathds{1}}

\title{A note on simple randomly switched linear systems}
\author{Gabriel LAGASQUIE}

\begin{document}

\maketitle

\begin{abstract}

We construct a planar process that switches randomly between the flows of two linear systems built from two Hurwitz matrices 
(all eigenvalues have negative real parts). The goal here is to study the long time behaviour according to the switching rates. 
We will see that, even if the two systems are stable, it is possible to obtain a blow up if we choose the switching rates wisely. 
Finally we will see a connection, between the tail of the invariant measure (when the switching times follow an 
exponential law) and the existence of a deterministic control that makes the process explode.

\end{abstract}


\section{Introduction}

Piecewise Deterministic Markov Processes (PDMPs) is a huge class of stochastic processes introduced by Davis \cite{dav}. In this note we are studying 
a particular subclass of these processes. we consider $(X_t, I_t)_{t\geq0}$ evolving on the space $\dR^d \times \{0,1\}$ solving the equation:
\[
 \dot{X}_t = F(X_t, I_t)
\]
where $F$ is a smooth function and $(I_t)$ is a jump process on $\{0,1\}$. We call $\lambda_i$ the jump rate of the process 
$(I_t)$ from state $i$ to state $1-i$. Such processes play a role in modeling problems in various fields such as 
molecular biology \cite{radu}, population dynamics \cite{lotka}, Internet traffic \cite{internet}. Even if these processes are easy 
to define, their stability is never clear.

More precisely, we are here interested in the long time qualitative behaviour of the planar randomly switched process $(X_t,I_t)_{t\geq0} \in \dR^d \times \{0,1\}$ 
solving the equation
\begin{equation}\label{eq:ode}
 \left\{
\begin{aligned}
 &\dot{X}_t = A_{I_t}X_t \\
 &X_0 = x_0,
\end{aligned}
\right.
\end{equation}
where $A_0$, $A_1$ are $d\times d$ real matrices, and $(I_t)$ is a Markov process on $E=\{0,1\}$ with constant jump rates:
\begin{align*}
 &\lambda_0 = \beta u,\\
 &\lambda_1 = \beta (1-u)
\end{align*}
where $\beta >0$ and $u\in(0,1)$. The process $((X_t,I_t))_{t\ge0}$ is a piecewise deterministic Markov process as defined in \cite{dav} and 
several examples have been studied in \cite{matt} or \cite{mal2}. It appears that the following quantity dictates the large time behaviour of this process according to the parameters $A_0$, 
$A_1$, $\beta$ and $u$:
\[
 \chi := \chi(A_0,A_1,\beta,u) = \underset{t \rightarrow +\infty}{\lim} \frac{1}{t} \log\left( \|X_t\| \right) \quad\text{a.s}.
\]
Indeed if $\chi$ is positive then the system explodes almost surely whereas if $\chi$ is negative the system 
is stable and goes to zero almost surely.\\

The asymptotic behaviour of such processes when $(I_t)$ is any deterministic process and $d=2$ is now understood. Given two matrices in $\cM_2(\dR)$, 
\cite{deter} gives a necessary and sufficient condition for the existence of a deterministic function $(I_t)$ that produces a blow up. 
But the articles \cite{matt} and \cite{mal1} have shown that the behaviour of planar randomly switched systems may be surprising: according to the 
value of the switching rate $\beta$, the process can either blow up or go to zero even if each system is stable. In \cite{matt}, 
the authors build a randomly switched system living in the space $\dR^{2n}$ which behaviour comes down to alternating 
blow up periods and stable periods a finite number of times as the jump rate grows.\\

A question raised from these works is the existence of a planar randomly switched system which behaviour is also alternating between 
blow up periods and stable periods. The goal of this note is to construct such a process.\\
Our idea is the following: put a stochastic control in a deterministic planar switched system which satisfies the qualitative long time behaviour 
condition (using the criteria in \cite{deter}).
For $a>0$ and $b>1$, we define two real matrices as follow:
\[
A_0=\begin{pmatrix}
 -a & b\\
 -\frac{1}{b} & -a
\end{pmatrix}
\quad\text{and}\quad
A_1=\begin{pmatrix}
 -a & \frac{1}{b}\\
 -b & -a
\end{pmatrix}.
\]

Thanks to these two matrices, we define the continuous process $((X_t,I_t))_{t\ge0} \in \dR^2\times\{0,1\}$ solution of \eqref{eq:ode}.
The path of $X_t$ follows a spiral clockwisely during a random time with exponential law of parameter $\lambda_i$ 
before switching on an other spiral as described in Figure \ref{fi:traj}.

\begin{figure}
\begin{center}
 \includegraphics[scale=0.6]{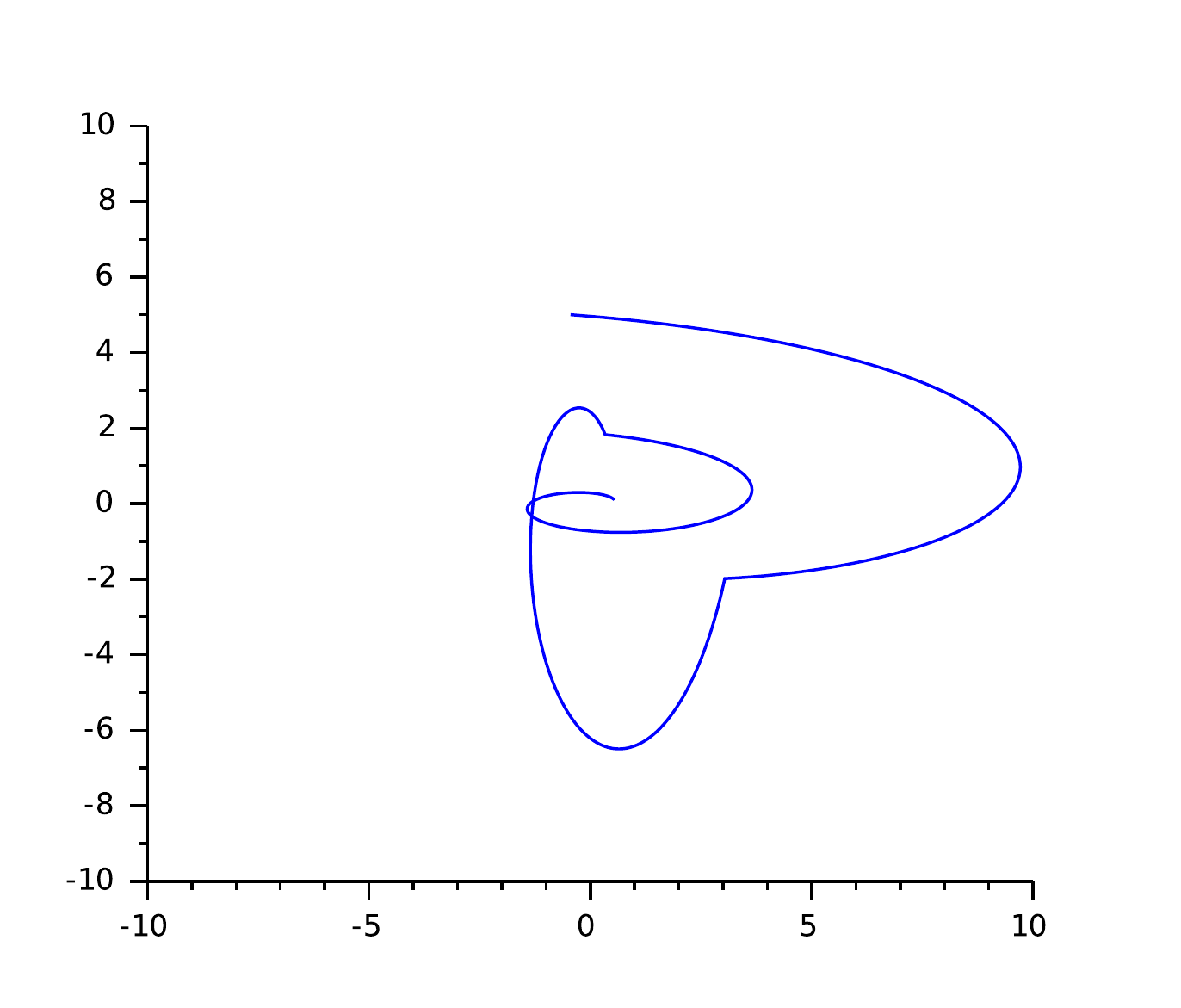}
 \caption{Example of a trajectory of the process \eqref{eq:ode}.}
 \label{fi:traj}
\end{center}
\end{figure}

We also define:
\[
 \chi := \chi(a,b,\beta,u) = \underset{t \rightarrow +\infty}{\lim} \frac{1}{t} \log\left( \|X_t\| \right) \quad\text{a.s}
\]
which sign gives the long time qualitative behaviour of our process as said before.\\

The main result of this note is the following which is similar to the main result in \cite{matt}:
\begin{thm}
For any $u \in (0,1)$,
\begin{align*}
&(i) \quad \forall a,b >0, \quad \chi(a,b,\beta,u) \longrightarrow -a \text{ when } \beta \rightarrow 0 \text{ or } \beta \rightarrow +\infty. \\
&(ii) \quad \text{For every } \beta,u>0 \text{ there exist }a>0 \text{ and } b>0 \text{ so that }  \chi(a,b,\beta,u)>0.
\end{align*}
\end{thm}

In Section 2, we establish an explicit expression of the function $\chi$. In Section 3 we use it to prove Theorem 1.1.

The explicit expression of $\chi$ obtained in Proposition 2.2 allows us to calculate numerically the function $\chi$ at $u$ constant (see Figure \ref{fi:prof}). It 
is also interesting to show the sign of the function $\chi$ according to $\beta$ and $u$ as we can see in Figure \ref{fi:patate}. In this figure, we can 
see that for $(\beta,u)$ in a compact, our system explodes when $t$ goes to $+\infty$. We can see the same behaviour in the process studied in 
\cite{matt}. But this is different from what we can observe in \cite{lotka} and \cite{mal1} where the explosion happens for $\beta$ large enough.

\begin{figure}
\begin{center}
 \includegraphics[scale=.8]{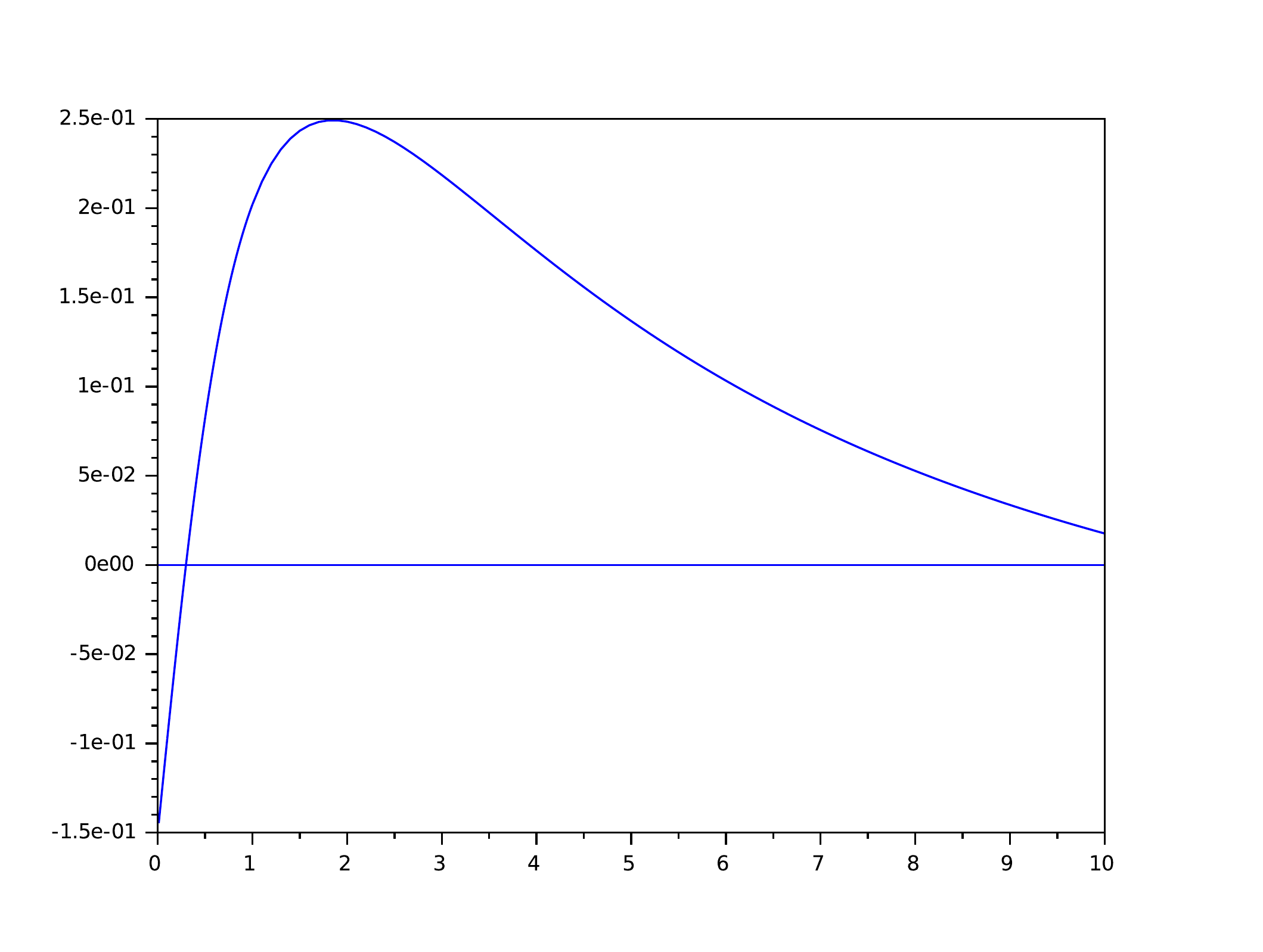}
 \caption{The function $\chi$\text{ for }$a=0.15$\text{, }$b=3$\text{ and }$u=0.5$.}
 \label{fi:prof}
\end{center}
\end{figure}

\begin{figure}
\begin{center}
 \includegraphics[scale=.5]{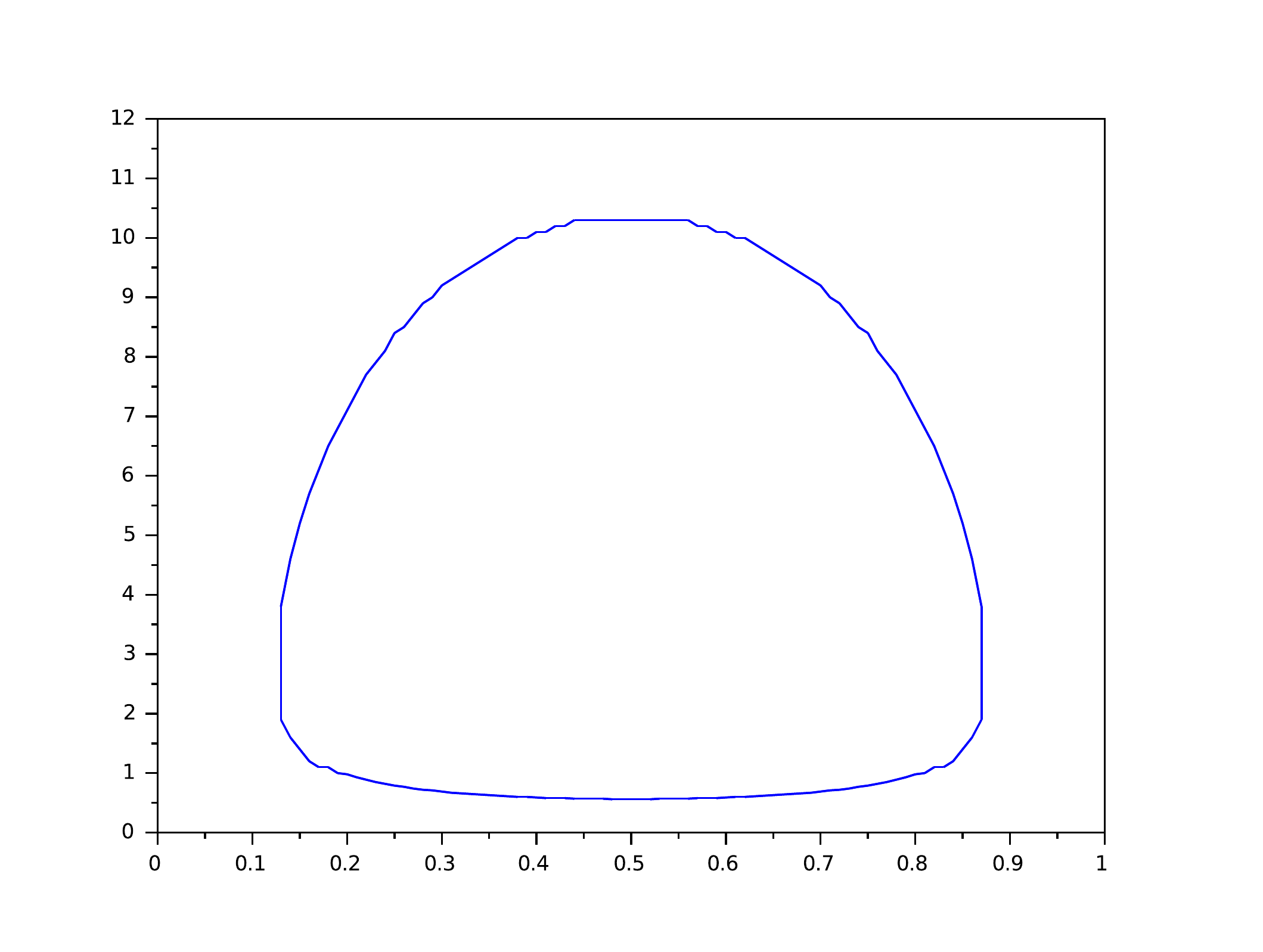}
 \caption{Sign of the function $\chi$\text{ for }$a=0.10$\text{, }$b=2.5$. Positive inside and negative outside.}
 \label{fi:patate}
\end{center}
\end{figure}

This graphics combined to Theorem 1.1 shows us that this process does not meet our expectations. There is a unique area of blow up 
whereas the deterministic periodic process has several blow up periods (see Section 4). It 
seems that, in the case of the random process, the variance of the waiting time in each state is too high to select precisely the good jump times 
that would allow it to have several blow up periods.\\

We will see in Section 5 a simple way to lower the variance of the waiting time without complicate too much the process, but first,
 in Section 4 we will give an interesting result about the deterministic process associated to our system.\\

In the last section we focus on the decentered switched system defined as follow:
\begin{equation}\label{eq:dec}
 \left\{
\begin{aligned}
 &\dot{X}_t = A_{I_t}(X_t-b_{I_t}) \\
 &X_0 = x_0,
\end{aligned}
\right.
\end{equation}
where $b_0$, $b_1$ are two different vectors of $\dR^d$.\\
There is a connection between the existence of an explosive trajectory (with a deterministic control) and 
the properties of the invariant measure, if it exists, in the random case (exponential switch times).

Let us call $T_0 = 0$, $T_1$, ..., $T_n$ the jumping times of the process $(X_t)$ and $\tau_0$, $\tau_1$, ..., $\tau_n$,  
the interjump times following an exponential law of parameters $\lambda_0$ and $\lambda_1$ alternatively. Let us define $Y_n = X_{T_{2n}}$. By simply solving the equation \ref{eq:dec} between each 
jumping times we obtain:
\[
 Y_n = e^{\tau_{n+1} A_1}e^{\tau_n A_0} Y_{n-1} - e^{\tau_{n+1}} b_0 + b_1.
\]
This kind of stochastic equation have been the object of numerous research. Applying Kesten's renewal theorem to the discrete process $(Y_n)$ will give some informations about the tail of the invariant measure of $(X_t)$ 
(see \cite{Kesten1973} for the main article, see \cite{gerold1}, \cite{desa} for different formulations of the same theorem).
\begin{thm}
If there is no deterministic control that makes the deterministic process associated to $(Y_t$) explodes, it means that there 
exists an invariant measure $\mu$ for our process and the support of $\mu$ is necessarily bounded.

In the other hand, we assume that $\chi < 0$ and that there exists a deterministic control such as:
\begin{equation}\label{eq:expl}
 \exists k\in 2\dN+1 \text{ such as }\forall x \in \dS_{d-1}, \exists t_0,t_1,...,t_k > 0 \text{ such as } \|e^{t_kA_{I_k}}...e^{t_1A_{I_1}}e^{t_0A_{I_0}}x\| > \|x\|.
\end{equation}
We call $B = e^{\tau_kA_{I_k}}...e^{\tau_1A_{I_1}}e^{\tau_0A_{I_0}}$ and 
$B_0$, $B_1$, ..., $B_n$ $n$ random variables i.i.d following the law of $B$. We also assume that:
\begin{equation}\label{eq:ret}
 \underset{n\geq0}{\max} \text{ } \dP \left( \frac{B_n...B_0x}{\|B_n...B_0x\|} \in U \right) >0 \text{ for all } x\in \dS_{d-1} \text{ and any open } \emptyset \neq U \subset \dS_{d-1}.
\end{equation}
\begin{equation}\label{eq:dens}
 \dP(B_{n_0}...B_0 \in .)\geq \gamma_0 \ind_{B_c(\Gamma_0)}\lambda \text{ for some } \Gamma_0 \in GL_d(\dR), \text{ } n_0 \in \dN \text{ and } c,\gamma_0>0.
\end{equation}
where $\lambda$ is the Lebesgue measure on $\dR^{n^2}$.
Under these assumptions, there exists a unique invariant measure $\mu$ for the process $(Y_t)$ and it has an heavy tail:
\[
\exists p_1>0 \text{ such as } \forall 0\leq p<p_1 \quad \dE_{\mu}[\|Y\|^p] < +\infty \text{ and } \forall p>p_1 \quad \dE_{\mu}[\|Y\|^p] = +\infty.
\]

\end{thm}

\section{Study of the angular process}
In order to study our process $((X_t,I_t))_{t\ge0}$ we use the polar coordinates for $X_t$. As a start, and as in 
\cite{matt}, let us look at the case of the deterministic process. Let $A$ be a 2-sized squared matrix and 
$x\in \dR^2\backslash\{0\}$. We set $(x_t)_{t\ge0}$ the solution of the ODE:
\[
\left\{
\begin{array}{ll}
        \dot{x}_t=Ax_t \quad\text{for all }t>0,\\
        x_0=x.
    \end{array}
\right.
\]
If $x$ is nonzero, then it is also true for any $x_t$ with $t$ positive. So we can define the polar coordinates 
$(r_t,\theta_t)$ of $x_t$. We call $e_\theta$ the unitary vector $(\cos\theta, \sin\theta)$ and $u_t = e_{\theta_t}$. Then 
$x_t = r_tu_t$. As $r_t^2=\DP{x_t,x_t}$, we obtain the relations:
\[
r_t\dot{r}_t=\DP{x_t,Ax_t}
\]
\[
A(r_tu_t) = \dot{x}_t = \dot{r}_tu_t+r_t\dot{u}_t.
\]
And then we have:
\begin{equation}
\dot{r}_t = r_t\DP{u_t,Au_t}
\end{equation}
\begin{equation}
\dot{u}_t = Au_t-\DP{u_t,Au_t}u_t.
\end{equation}
Let us write the equation (2) with the angle $\theta_t$. As $\dot{u}_t = \dot{\theta}_te_{\theta_t+\pi/2}$, by making the scalar 
product of (2) with $e_{\theta_t+\pi/2}$, we obtain:
\begin{equation}
\dot{\theta}_t=\DP{Ae_{\theta_t},e_{\theta_t+\pi/2}}.
\end{equation}

We use now the polar coordinates in order to study the process $((X_t,I_t))_{t\ge0}$. Between the jumps, the process 
follows the flow determined by $A\in\{A_0,A_1\}$. From Equation (3), the development of $\theta$ 
is deterministic and does not depend on $r$. As a consequence, the processes  $((\Theta_t,I_t))_{t\ge0}$ and $((U_t,I_t))_{t\ge0}$ 
are piecewise deterministic Markov processes on $\dR\times\{0,1\}$ and $\dS_1\times\{0,1\}$ respectively. The principal interest of the study 
of these processes lies in the fact that the development of $(R_t)_{t\ge0}$ is determined by that of the process $((\Theta_t,I_t))_{t\ge0}$ 
as shown by Equation (1). Indeed, by solving (1) between the jumps and by calling $\cA(\theta,i)=\DP{A_ie_{\theta},e_{\theta}}$, 
we obtain:
\begin{equation}
R_t = R_0\exp\left(\int_0^t \cA(\Theta_s,I_s)ds\right).
\end{equation}

So, as suggested by Equation (4), in order to study the behaviour of $(R_t)_{t\ge0}$ we are going to study the process 
$((\Theta_t,I_t))_{t\ge0}$ and particularly its invariant measure.
\begin{lem}
 The invariant measure $\mu$ of the process $((\Theta_t,I_t))_{t\ge0}$ is given by,
\[
 \mu(d\theta,i) = \rho_i(\theta)d\theta
\]
where:
\begin{align*}
& \rho_0 = \frac{\Phi}{d_0} \quad \text{and} \quad \rho_1 = \frac{C-\rho_0 d_0}{d_1}; \\
& d_0(\theta)=-b\sin^2(\theta)-\frac{1}{b}\cos^2(\theta) \quad\text{and}\quad d_1(\theta)=-\frac{1}{b}\sin^2(\theta)-b\cos^2(\theta); \\
& \Phi(\theta) = \left(K + \int_0^\theta \beta C (1-u)\frac{1}{d_1(\alpha)} e^{-\beta v(\alpha)}d\alpha \right) e^{\beta v(\theta)}; \\
& v \text{ is the primitive null in zero of the function }  -(\frac{u}{d_0}+\frac{1-u}{d_1}); \\
& K \text{ and } C \text{ are two explicit constants.}
\end{align*}
\end{lem}
\begin{proof}

We define, for $i\in\{0,1\}$,
\[
d_i(\theta) = \DP{A_ie_{\theta},e_{\theta+\pi/2}}.
\]
We call $L$ the infinitesimal generator associated to the semi-group $P_t$ of the process $((\Theta_t,I_t))_{t\ge0}$. Remind that 
$P_t = E[f(\Theta_t,I_t)|\Theta_0 = \theta,I_0=i]$. $L$ is given by:
\begin{equation}
Lf(\theta,i) = d_i(\theta)\frac{\partial f}{\partial\theta}(\theta,i)+\lambda_i\left(f(\theta,1-i)-f(\theta,i)\right)
\end{equation}
for any smooth function $f$.\\
In our study case, the functions $d_i$ are given by,
\[
d_0(\theta)=-b\sin^2\theta-\frac{1}{b}\cos^2\theta \quad\text{and}\quad d_1(\theta)=-\frac{1}{b}\sin^2\theta-b\cos^2\theta.
\]
We notice that the functions $d_i$ are negative. This obviously implies that the process $((U_t,I_t))_{t\ge0}$ is recurrent, irreducible
 and it has a unique invariant measure. We want to write it 
as follows:
\begin{equation}
\mu(d\theta,i) = \rho_i(\theta)d\theta .
\end{equation}
where $\rho_0$ and $\rho_1$ are two $2\pi$-periodical continuous functions.\\
In this case, for every good functions $f$ defined on $\dS_1\times\{0,1\}$, we will have:
\[
\int_{\dS_1\times\{0,1\}} L f(\theta,i) d\mu(\theta,i) = 0.
\]
Let $f$ be a $\cC^1$ function defined on $\dS_1\times\{0,1\}$. Injected in the previous formula, the expression (6) allows us 
to obtain:
\begin{equation}
\int_0^{2\pi} L f(\theta,0)\rho_0(\theta)d\theta + \int_0^{2\pi} L_\beta f(\theta,1)\rho_1(\theta)d\theta = 0.
\end{equation}
Assume initially that $\forall i$, $f(\theta,i)=f(\theta)$. So, by injecting in (6), and after integrating by parts, we obtain:
\begin{equation}
d_0\rho_0 + d_1\rho_1 = C
\end{equation}
where $C$ is a negative constant depending only of the parameters of the problem $a$, $b$, $u$ and $\beta$.\\
Assume now that $f(\theta,0)=f(\theta)$ and $f(\theta,1)=0$. The same way, we obtain the relation:
\[
(d_0\rho_0)'=(1-u)\beta\rho_1-u\beta\rho_0.
\]
Set $\Phi = d_0\rho_0$. By using the relation (7), we can write the last equation as a linear differential equation:
\begin{equation}
\Phi' = -\beta\Phi(\frac{u}{d_0}+\frac{1-u}{d_1}) + \frac{\beta C (1-u)}{d_1}.
\end{equation}
This differential equation is easily solved, we obtain:
\[
\Phi(\theta) = \left(K + \int_0^\theta \beta C (1-u)\frac{1}{d_1(\alpha)} e^{-\beta v(\alpha)}d\alpha \right) e^{\beta v(\theta)}
\]
where $K$ is an integration constant and $v$ is the primitive null in zero of the function  $-(\frac{u}{d_0}+\frac{1-u}{d_1})$. 
We can calculate $v$ explicitly on the interval $(-\frac{\pi}{2},\frac{\pi}{2})$:
\[
v(\theta) = u\arctan(b\tan(\theta))+(1-u)\arctan(\frac{1}{b}\tan(\theta)).
\]
We extend $v$ to $\dR^+$ by continuity.\\
The constants $C$ and $K$ are left to calculate in such a way that $\mu$ is a probability measure. So:
\[
\text{(i)  }\rho_0 \text{ and } \rho_1 \text{ are non negative,} \quad
\text{(ii) }\int_0^{2\pi}\rho_0(\theta)d\theta + \int_0^{2\pi}\rho_1(\theta)d\theta= 1.
\]
The condition (i) gives us the following condition on $\Phi$: $C<\Phi<0$. By developping this inequality, we obtain:
\[
\forall \theta \in \dR^+, \text{ }e^{-\beta v(\theta)}-\int_0^\theta \beta(1-u)\frac{1}{d_1(\alpha)} e^{-\beta v(\alpha)}d\alpha > \frac{K}{C} > -\int_0^\theta \beta(1-u)\frac{1}{d_1(\alpha)} e^{-\beta v(\alpha)}d\alpha .
\]
As $\theta \longrightarrow +\infty$, we obtain:
\[
\frac{1}{\kappa}=\frac{K}{C}= -\beta(1-u)\int_0^\infty \frac{1}{d_1(\alpha)} e^{-\beta v(\alpha)}d\alpha.
\]
In order to find $K$, we use the condition (ii) and we get easily:
\[
\frac{1}{K} = \int_0^{2\pi} \left[ e^{\beta v(\theta)} \left( 1 + \kappa\beta(1-u) \int_0^\theta \frac{1}{d_1(\alpha)} e^{-\beta v(\alpha)}d\alpha \right)\left(\frac{1}{d_0(\theta)}-\frac{1}{d_1(\theta)}\right)+\kappa \frac{1}{d_1(\theta)}\right]d\theta.
\]
We finally get expressions of $\rho_0$ and $\rho_1$. Reciprocally, we check that the functions we obtained are solutions of our 
problem (one has to check the $2\pi$-periodicity of these functions).
\end{proof}

We want now to be able to describe the stability of our switched process according to the jump parameters $u$ and $\beta$, and to the parameters 
of our matrices $a$ and $b$. It appears thanks to Formula (4) that the following quantity:
\[
 \chi(\beta) = \underset{t \rightarrow +\infty}{\lim} \frac{1}{t}\log\left(\frac{R_t}{R_0}\right) = \underset{t \rightarrow +\infty}{\lim} \frac{1}{t} \left(\int_0^t \cA(\Theta_s,I_s)ds\right)
\]
is worth studying because not only its sign gives the behaviour in large time of our process but also, thanks to the Ergodic 
Theorem and the previous Lemma, we will obtain an explicit expression of $\chi$.
\begin{prop}
 The function $\chi$ can be calculated explicitly according to the data of the problem:
\[
 \chi = -a - (1-u) \beta \kappa K \frac{(b-\frac{1}{b})}{2} \int_0^{2\pi} \sin(2\theta)\left(\frac{1}{d_0(\theta)}+\frac{1}{d_1(\theta)}\right)e^{\beta v(\theta)} \left(\int_\theta^\infty \frac{e^{-\beta v(\alpha)}}{d_1(\alpha)} d\alpha\right)d\theta.
\]
\end{prop}

\begin{proof}
 Since we have:
\[
 \frac{1}{t} \log (\frac{R_t}{R_0}) = \frac{1}{t} \left(\int_0^t \cA(\Theta_s,I_s)ds\right).
\]
The Ergodic Theorem tells us that (we identify $\theta$ with $e_{\theta}$):
\[
\frac{1}{t} \left(\int_0^t \cA(\Theta_s,I_s)ds\right) \underset{t\longrightarrow \infty}{\longrightarrow} \int_0^{2\pi} \cA(\theta,i)d\mu(\theta,i) = \chi.
\]
Lemma 2.1 gives us the explicit formulation of the invariant measure of the process $((\Theta_t,I_t))_{t\ge0}$ which allows us 
to obtain the claimed expression for $\chi$.
\end{proof}

\section{Stability of the switched process}

In the last section, we obtained the explicit expression of $\chi$. We are now going to give some results on some of the asymptotic values of 
$\chi$.
\begin{thm}
 For any $u \in (0,1)$,
\begin{align*}
&(i) \quad \forall a,b >0 \quad \chi(a,b,\beta,u) \longrightarrow -a \text{ when } \beta \rightarrow 0 \text{ or } \beta \rightarrow +\infty. \\
&(ii) \quad \text{For every } \beta,u>0 \text{ we can choose }a>0 \text{ and } b>0 \text{ so that }  \chi(a,b,\beta,u)>0.
\end{align*}
\end{thm}

\begin{proof}
(i) One can find an other way to prove this result in \cite{matt} in a more general way.\\
In order to study the limit of $\chi$ in $0$, we study the limit of $-\frac{1}{\kappa} = \beta (1-u)\int_0^\infty \frac{e^{-\beta v(\alpha)}}{d_1(\alpha)}d\alpha$ in $0$. 
We notice that:
\begin{align*}
\int_0^{+\infty} \frac{e^{-\beta v(\alpha)}}{d_1(\alpha)}d\alpha &= \sum_{i=0}^{+\infty} \int_{i\pi}^{(i+1)\pi} \frac{e^{-\beta v(\alpha)}}{d_1(\alpha)}d\alpha \\
								  &= \sum_{i=0}^{+\infty} e^{-\beta i \pi} \int_{0}^{\pi} \frac{e^{-\beta v(\alpha)}}{d_1(\alpha)}d\alpha \text{ because }v(\alpha +\pi)=v(\alpha)+\pi\\
								  &= \frac{1}{1-e^{-\beta \pi}} \int_{0}^{\pi} \frac{e^{-\beta v(\alpha)}}{d_1(\alpha)}d\alpha.
\end{align*}
Multiplicating by $\beta (1-u)$ and observing the presence of a rate of increase we get:
\[
\beta(1-u) \int_0^\infty \frac{e^{-\beta v(\alpha)}}{d_1(\alpha)}d\alpha \longrightarrow \frac{1-u}{\pi} \int_0^\pi \frac{1}{d_1(\alpha)}d\alpha = l.
\]
We finally can say that:
\[
\forall \theta,\quad\left( \beta(1-u) \int_\theta^\infty \frac{e^{-\beta v(\alpha)}}{d_1(\alpha)}d\alpha \right) \underset{\beta \longrightarrow 0}{\longrightarrow} l.
\]
Moreover, as:
\[
\int_0^{2\pi} \sin(2\theta)\left(\frac{1}{d_0(\theta)}+\frac{1}{d_1(\theta)}\right)d\theta = 0
\]
and the product $\kappa K$ is bounded as $\beta \rightarrow 0$, we deduce that $\chi(\beta) \longrightarrow 0$ when $\beta$ goes to $0$.\\
The limit of $\chi$ at $+\infty$ is obtained using the Laplace method:
\[
\forall \theta,\text{  } \beta(1-u) \int_\theta^\infty \frac{e^{-\beta v(\alpha)}}{d_1(\alpha)}d\alpha \underset{+\infty}{\sim} \frac{(1-u)e^{-\beta v(\theta)}}{d_1(\theta)v'(\theta)}.
\]
From this relation, we obtain that $\kappa$ and $K$ converge when $\beta$ goes to $+\infty$ and, by dominated convergence,
\[
\int_0^{2\pi} \sin(2\theta)\left(\frac{1}{d_0(\theta)}+\frac{1}{d_1(\theta)}\right)e^{\beta v(\theta)} \beta(1-u) \left(\int_\theta^\infty \frac{e^{-\beta v(\alpha)}}{d_1(\alpha)} d\alpha\right)d\theta \underset{\beta \rightarrow +\infty}{\longrightarrow} 0.
\]
So we finally conclude that $\chi$ goes to $-a$ when $\beta$ goes to $+\infty$.\\
(ii) We will now show that for any parameter $\beta$, we can choose $b$ large enough and $a$ small enough so that the process goes to $+\infty$
almost surely (ie $\chi(\beta) > 0$). We can write the function $\chi$ as follows:
\[
\chi(\beta) = -a + K \int_0^{\pi} f(\theta) g(\theta) d\theta
\]
because the function under the integral is $\pi$-periodic and where:
\begin{align*}
& f(\theta) = (b-\frac{1}{b}) \sin(2\theta) \left( \frac{1}{d_0(\theta)}+\frac{1}{d_1(\theta)} \right),\\
& g(\theta) = e^{\beta v(\theta)} \frac{\int_{\theta}^{+\infty} \frac{e^{-\beta v(\alpha)}}{d_1(\alpha)} d\alpha}{\int_{0}^{+\infty} \frac{e^{-\beta v(\alpha)}}{d_1(\alpha)} d\alpha}.
\end{align*}
Finally we write $\chi$ as follows:
\[
 \chi(\beta) = -a + K \int_0^{\frac{\pi}{2}} f(\theta) \left(g(\theta) - g(\pi-\theta)\right) d\theta
\]
A straightforward computation leads to:
\[
 \int_0^{\frac{\pi}{2}} f(\theta)d\theta = \frac{4}{\gamma} \frac{b^2-\frac{1}{b^2}}{(b-\frac{1}{b})^2}  \log \left( \abs{\frac{\gamma -1}{\gamma +1} } \right)
\]
where $\gamma = \sqrt{1+\frac{4}{(b-\frac{1}{b})^2}}$. We then conclude that:
\[
 \int_0^{\frac{\pi}{2}} f(\theta)d\theta \longrightarrow -\infty \text{ when } b \text{ goes to}+\infty.
\]
In order to determine the limit of $g$ when $b$ goes to $+\infty$ we rewrite $g$ as follows using the same method as before:
\begin{align*}
g(\theta) &= e^{\beta v(\theta)} \left( 1 - (1-e^{-\beta \pi}) \frac{\int_0^{\theta} \frac{1}{d_1(\alpha)} e^{-\beta v(\alpha)}d\alpha}{\int_0^{\pi} \frac{1}{d_1(\alpha)} e^{-\beta v(\alpha)}d\alpha} \right)\\
&= e^{\beta v(\theta)} \left( 1 - (1-e^{-\beta \pi}) \frac{\int_0^{\theta} \frac{1}{\cos^2(\alpha)+\frac{1}{b^2}\sin^2(\alpha)} e^{-\beta v(\alpha)}d\alpha}{\int_0^{\pi} \frac{1}{\cos^2(\alpha)+\frac{1}{b^2}\sin^2(\alpha)} e^{-\beta v(\alpha)}d\alpha} \right).
\end{align*}
Using the fact that:
\begin{align*}
&v(\theta) \underset{b \rightarrow +\infty}{\longrightarrow}  \frac{\pi}{4} \text{ if }\theta \in (0,\frac{\pi}{2})\\
&v(\theta) \underset{b \rightarrow +\infty}{\longrightarrow}  \frac{3\pi}{4} \text{ if }\theta \in (0,\frac{\pi}{2}),\\
\end{align*}
it is now easy to see that:
\begin{align*}
&g(\theta) \underset{b \rightarrow +\infty}{\longrightarrow} e^{\beta \frac{\pi}{4}} \text{ if }\theta \in (0,\frac{\pi}{2})\\
&g(\pi-\theta) \underset{b \rightarrow +\infty}{\longrightarrow} e^{-\beta \frac{\pi}{4}} \text{ if }\theta \in (0,\frac{\pi}{2}).\\
\end{align*}
Since $K<0$, we can conclude that for $b$ large enough,
\[
 K\int_0^{\frac{\pi}{2}} f(\theta) \left(g(\theta) - g(\pi-\theta)\right) d\theta > 0.
\]
So for an appropriate choice of $a$, we see that $\chi(\beta) > 0$ inducing that the process explodes almost surely for this $\beta$ and $a$ with $b$ large enough.
\end{proof}

\section{A quick look at the associated deterministic process}

The deterministic switched system can be introduced as follows:
\begin{equation}\label{eq:ode2}
\begin{cases}
x_0 \in \dR^2\backslash{(0,0)}\\
\dot{x_t} = A_{v_t} x_t
\end{cases}
\end{equation}
where $(v_t)\in \cF(\dR_+,\{0,1\})$ is the control function. The behaviour of these systems is well known, see \cite{deter} for more details. 
Consequently, for our choice of matrices, we know that there exist a worst trajectory, that is to say, a choice of the control function 
$v$ that makes the system explode, and this control is periodic.\\

Let us define $(v_t)$ a $\frac{1}{u(1-u)\beta}$-periodical process satisfying:
\begin{align*}
&\forall t \in [0,\frac{1}{u\beta}), \quad v_t = 0\\
&\forall t \in [\frac{1}{u\beta},\frac{1}{u(1-u)\beta}), \quad v_t = 1.
\end{align*}
We denote $\chi^d(\beta) = \underset{t\longrightarrow \infty}{\text{lim}} \frac{1}{t}\log \left(\frac{\|x_t\|}{\|x_0\|}\right)$.\\
By simply solving the equation for the process $(x_t)$ and calculating the matrices exponential we obtain:
\[
x_t = e^{-at}B_{\xi_{N(t)+1}}(a_t)B_{\xi_{N(t)}}(\tau_{N(t)})\cdots B_{\xi_1}(\tau_1)
\]
where 
\[
B_0(t) = e^{tA_0} = \begin{pmatrix}
 \cos(t) & b\sin(t)\\
 -\frac{1}{b}\sin(t) & \cos(t)
\end{pmatrix}
\quad\text{and}\quad
B_1(t) = e^{tA_1} = \begin{pmatrix}
 \cos(t) & \frac{1}{b}\sin(t)\\
 -b\sin(t) & \cos(t)
\end{pmatrix},
\]
$N(t)$ is the number of jumps of the deterministic 
process before $t$ and $\tau_1$, $\tau_2$, ...,$\tau_{N(t)}$ the time spent in the states $\xi_1$, $\xi_2$, ..., $\xi_{N(t)}$, $\xi_{N(t)+1}$ by $v(t)$.\\
We have the surprising following result:
\begin{prop}
 $\chi^d(\beta)$ almost surely does not depend on the initial value $x_0$.
\end{prop}
\begin{proof}
Although this result is true for any $u\in(0,1)$, for sake of simplicity, we will prove it only for the case $u=\frac{1}{2}$ where the computation is not heavy. 
In order to prove this result, we first notice that:
\[
 \frac{1}{2nt} \log \left( \frac{\|[B_0(\tau)B_1(\tau)]^nX_0\|}{\|X_0\|} \right) \underset{n \rightarrow +\infty}{\longrightarrow} \chi_d(\beta)
\]
where $\tau = \frac{2}{\beta}$. Let us calculate the matrix:
\[
 B_0(\tau)B_1(\tau) = \begin{pmatrix}
 \cos^2(\tau)-b^2 \sin^2(\tau) & C\sin(2\tau)\\
 -C\sin(2\tau) & \cos^2(\tau)-\frac{1}{b^2} \sin^2(\tau)
\end{pmatrix}
\]
where $C=\frac{1}{2}(b+\frac{1}{b}) >1$. We calculate the characteristic polynomial of this matrix and we obtain:
\[
 X^2 + X (4C^2 \sin^2(\tau)-2) +1 = 0.
\]
By a simple analysis of this polynomial, we can show that there exists two real eigenvalues if $2C^2\sin^2(\tau)-1 > 0$ and two joint 
complex eigenvalues if $2C^2\sin^2(\tau)-1 < 0$, concluding that our matrix is diagonalisable in $\dC$. So there exists $P\in GL_2(\dC)$ 
such that:\
\[
 (B_1(\tau)B_0(\tau))^n = P^{-1} \begin{pmatrix}
 \lambda_1^n & 0\\
 0 & \lambda_2^n
\end{pmatrix} P.
\]
Let us now write:
\begin{align*}
 \frac{1}{2n\tau} \log \left( \frac{\|[B_0(\tau)B_1(\tau)]^nX_0\|}{\|X_0\|} \right) &= \frac{1}{2nt} \log \left( 
\| \begin{pmatrix}
 \lambda_1^n & 0\\
 0 & \lambda_2^n
\end{pmatrix} P X_0 \| \right) +o(1) \\
&= \frac{1}{2n\tau} \log \left( 
\| \begin{pmatrix}
 \lambda_1^n & 0\\
 0 & \lambda_2^n
\end{pmatrix} \begin{pmatrix}
x\\
y
\end{pmatrix} \| \right) +o(1).
\end{align*}
If the factor $4C^2 \sin^2(\tau)-2 \neq 0$ then one of the eigenvalues is superior to the other in absolute value, for example
 $|\lambda_1 | > |\lambda_2|$. In this case, if $x \neq 0$ ie $X_0 \notin E_{\lambda_2}$ the characteristic space of the eigenvalue $\lambda_2$:
\[
 \frac{1}{2n\tau} \log \left( \frac{\|[B_0(\tau)B_1(\tau)]^nX_0\|}{\|X_0\|} \right) = \frac{1}{2\tau} \log |\lambda_1| +o(1).
\]
If $x = 0$ i.e. $X_0 \in E_{\lambda_2}$:
\[
 \frac{1}{2n\tau} \log \left( \frac{\|[B_0(\tau)B_1(\tau)]^nX_0\|}{\|X_0\|} \right) = \frac{1}{2\tau} \log |\lambda_2| +o(1).
\]
If the factor $4C^2 \sin^2(\tau)-2 = 0$ then for every $X_0$:
\[
 \frac{1}{2n\tau} \log \left( \frac{\|[B_0(\tau)B_1(\tau)]^nX_0\|}{\|X_0\|} \right) = \frac{1}{2\tau} \log |\lambda_1| +o(1).
\]

\end{proof}

Figure \ref{fi:3} illustrates Proposition 4.1.

\begin{figure}
\begin{center}
 \includegraphics[scale=.8]{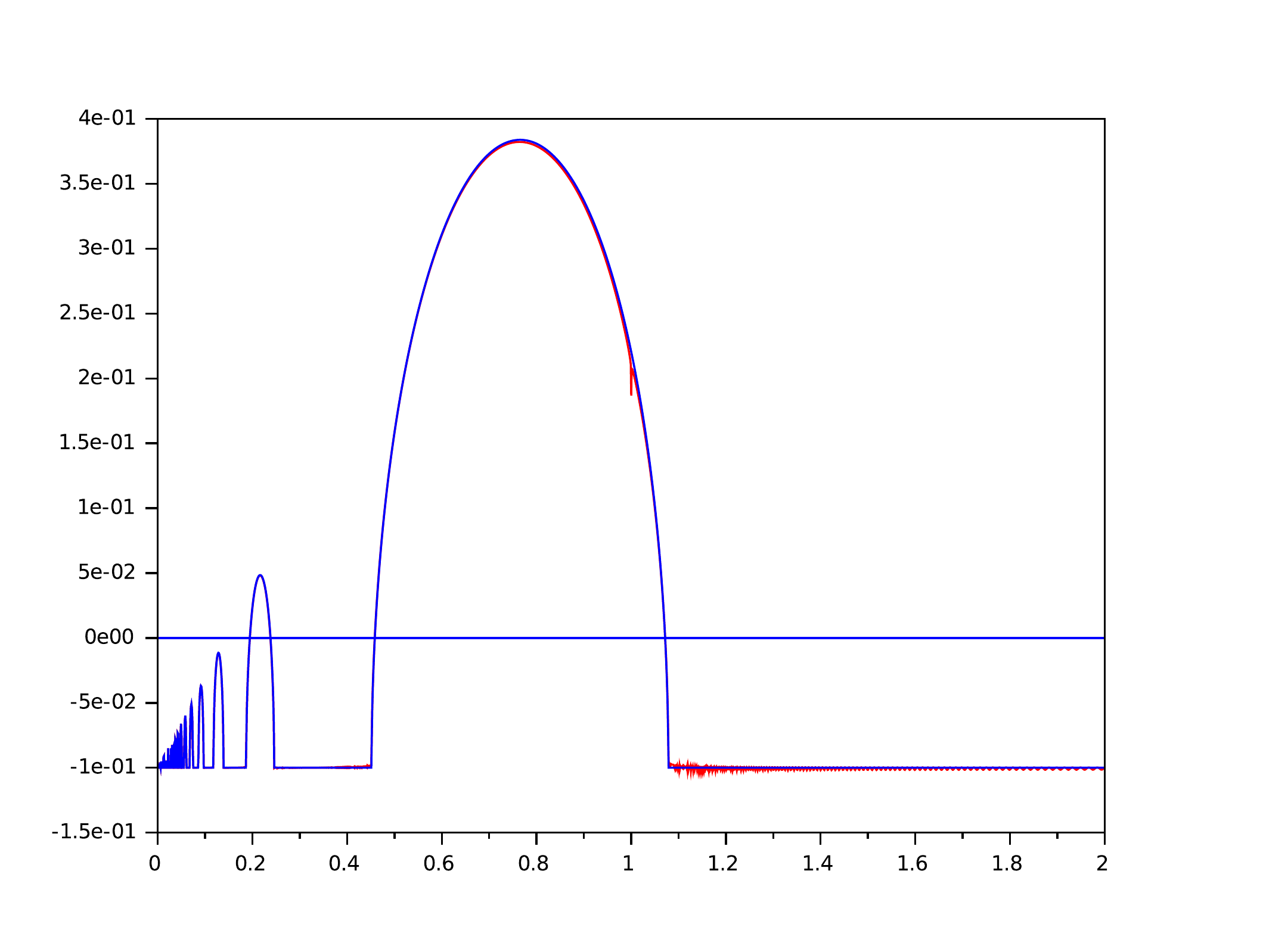}
 \caption{In blue the general shape of $\chi^d$ for $a=0.1$, $b=2$ and $u=0.5$.
\text{In red, }$x_0$\text{ is in }$E_{\lambda_2}$\text{ for }$\beta = 1$.}
\label{fi:3}
\end{center}
\end{figure}

\section{A way to obtain several blow up areas}

One of the expectations of this article was to get a planar linear switched system whose $\chi$ function admits several 
blow up areas. A quick look at the deterministic case (\ref{fi:3}) allowed us to believe that it could be possible with our system. 
Unfortunately the simulation we can see in Figure \ref{fi:prof} seems to show that we did not manage to succeed. We slightly modify the jump 
times in order to mimic the deterministic evolution \eqref{eq:ode2}.

Let $u=\frac{1}{2}$ (for sake of simplicity) and let $n$ be a stricly positive integer. We now define the piecewise deterministic Markov process $((X_t^n,I_t^n))_{t\ge0}$ as follows: $(I_t)_{t\geq 0}$ is a continuous time Markov 
process defined on the state space $\{0,1,2,...,2n-1\}$ with a constant jump rate $\frac{n\beta}{2}$ and a jump mesure:
\begin{align*}
&Q(i,.) = \delta_{i+1} \quad \text{if} \quad i<2n-1,\\
&Q(2n-1,.) = \delta_0. 
\end{align*}
And we define 
$(X_t^n)_{t\ge0}$ as the solution of:
\[
X_t = X_0^n + \int_0^t A_{\mathds{1}_{I_t < n}}X_s^nds.
\]
This process looks like \eqref{eq:ode}. The difference lies in the fact that, instead 
of waiting for an exponential time of parameter $\frac{\beta}{2}$ before jumping from a matrix to the other, the process waits for a time following a Gamma law of 
mean $\frac{2}{\beta}$ and variance $\frac{4}{n\beta^2}$.\\
The next result compares the stochastic process to the periodic one.\\

\begin{figure}
\begin{center}
 \includegraphics[scale=.8]{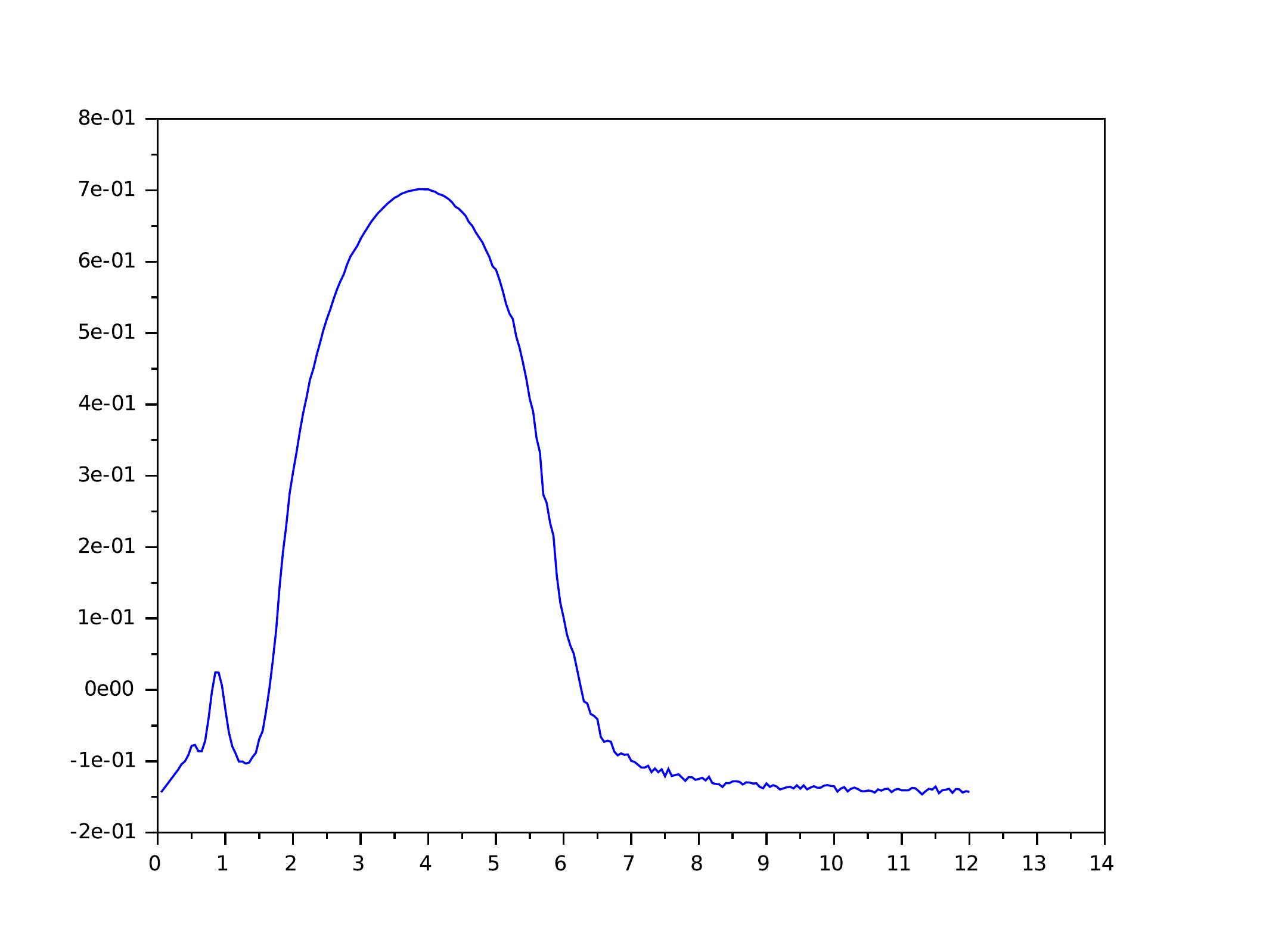}
 \caption{For $a = 0.15$, $b=3$ and $n=50$, a second blow up area appears in $\chi$.}
\label{fi:4}
\end{center}
\end{figure}

\begin{lem}
 For every $T>0$, the process $(X_t^n)_{0\leq t \leq T}$ uniformly converges to the deterministic process $(X_t)_{0\leq t \leq T}$ solution of \eqref{eq:ode2}.
\end{lem}

\begin{proof}
Let us call $N^n(t)$ the number of jumps before $t$ of the process $(I_t^n)$, $\tau_1^n$, $\tau_2^n$, ...,$\tau_{N(t)}^n$ are the interjump times, they follow a gamma law 
of mean $\frac{1}{\beta}$ and variance $\frac{1}{n\beta^2}$. We denote $a_t^n = t - \sum_{k=1}^{N(t)} \tau_k^n$. We finally call $\xi_1^n$, $\xi_2^n$, ...,$\xi_{N(t)+1}^n\in \{0,1\}$ 
the states.\\
By simply solving the equation for the process $(X_t^n)$ and calculating the matrices exponential we obtain:
\[
X_t^n = e^{-at}B_{\xi_{N^n(t)+1}^n}(a_t^n)B_{\xi_{N^n(t)}^n}(\tau_{N^n(t)}^n)\cdots B_{\xi_1^n}(\tau_1^n)
\]
where 
\[
B_0(t) = \begin{pmatrix}
 \cos(t) & b\sin(t)\\
 -\frac{1}{b}\sin(t) & \cos(t)
\end{pmatrix}
\quad\text{and}\quad
B_1(t) = \begin{pmatrix}
 \cos(t) & \frac{1}{b}\sin(t)\\
 -b\sin(t) & \cos(t)
\end{pmatrix}.
\]
Let us call $N(t)$ the number of jumps of the deterministic process. We have:
\[
X_t = e^{-at}B_{\xi_{N(t)+1}}(a_t)B_{\xi_{N(t)}}(\frac{1}{\beta})\cdots B_{\xi_1}(\frac{1}{\beta}).
\]
We form the difference:
\begin{align*}
\| X_t^n - X_t \| \leq\text{ } & \mathds{1}_{N^n(t)=N(t)} \| B_{\xi_{N(t)+1}^n}(a_t^n)\cdots B_{\xi_1^n}(\tau_1^n) - B_{\xi_{N(t)+1}}(a_t)\cdots B_{\xi_1}(\frac{1}{\beta}) \|\|X_0\| \\
                      +& \mathds{1}_{N^n(t)\neq N(t)} \|B_{\xi_{N^n(t)+1}^n}(a_t^n)\cdots B_{\xi_1^n}(\tau_1^n) - B_{\xi_{N(t)+1}}(a_t)\cdots B_{\xi_1}(\frac{1}{\beta}) \| \|X_0\| .
\end{align*}
We can prove by induction on $N(t)$ that:
\[
\| B_{\xi_{N(t)+1}^n}(a_t^n)\cdots B_{\xi_1^n}(\tau_1^n) - B_{\xi_{N(t)+1}}(a_t)\cdots B_{\xi_1}(\frac{1}{\beta}) \| \longrightarrow 0 \text{ in probability.}
\]
The case $N(t)=1$ shows us what happens:
\begin{align*}
\|B_{\xi_2}(t-\tau_1)B_{\xi_1}(\tau_1)-B_{\xi_2}(t-\frac{1}{\beta})B_{\xi_1}(\frac{1}{\beta})\| &\leq \|B_{\xi_2}(t-\tau_1)\| \|B_{\xi_1}(\tau_1)-B_{\xi_1}(\frac{1}{\beta})\| \\
                                                                                                &\text{ }+\|B_{\xi_2}(\frac{1}{\beta})-B_{\xi_2}(\tau_2)\|\|B_{\xi_1}(\tau_1)\|        \\
												&\leq C_1|\tau_1-\frac{1}{\beta}|.
\end{align*}
By using the inequality of Bienaymé-Tchebychev, we gain the claimed convergence. Moreover we can say that this convergence also holds 
in $L^1$ as the random variable is bounded because our process does not explode in a finite time.\\
We prove now that $\dP[N^n(t)=N(t)] \longrightarrow 1$ when $n$ goes to $+\infty$. First, let us write that $N(t) = p$ and $t=\frac{p}{\beta}+\eta$ where $0\leq\eta<\frac{1}{\beta}$. 
We have:
\begin{align*}
\dP[N^n(t)=p] &= \dP\left[\tau_1+\tau_2+...+\tau_p \leq t < \tau_1+...+\tau_p+\tau_{p+1}\right]         \\
		 &= \dP\left[0\leq t-\sum_{i=1}^p \tau_i \text{ and } t<\sum_{i=1}^{p+1} \tau_i\right]     \\
		 &= \dP\left[\sum_{i=1}^p \tau_i -\frac{p}{\beta} < \eta\right]\dP\left[t<\sum_{i=1}^{p+1} \tau_i \text{ }|\text{ }0\leq t-\sum_{i=1}^p \tau_i \right].
\end{align*}
The first probability goes to $1$ when $n$ goes to $+\infty$ thanks to the Bienaymé-Tchebychev inequality and the other probability is nonzero. 
This means that we have the claimed convergence: $\dP[N^n(t)=N(t)] \longrightarrow 1$ when $n$ goes to $+\infty$.\\
Back in the inequality, we take the expectation and using the previous results we 
prove that $\dE[\|X_t^n - X_t\|]$ goes to $0$ uniformly on $[0,T]$.\\
\end{proof}

\section{Behaviour of the switched system with two centers of attraction}

We slightly modify our initial system by shifting the centers of attraction of each flow.
Assume that $A_0$ and $A_1$ are Hurwitz matrices and $b_0$ and $b_1$ in $\dR^d$ (with $b_0 \neq b_1$). 
For $i=0,1$ and $x\in\dR^d$, let $t>0\mapsto \varphi_t^i(x)$ the solution of $\dot {x}_t=A_i(x_t-b_i)$ 
with $x_0=x$. 

\begin{rem}\label{rem:expo}
There exists $C\geq 1$ and $\eta>0$ such that, for all $t\geq 0$, 
\[
\NRM{\varphi_t^i(x)-b_i}\leq C e^{-\eta t}\NRM{x-b_i}.
\]
\end{rem}

The goal of this section is to discuss the long time behavior of $(X_t)$ if ${(I_t)}_{t\geq 0}$ jumps from 
1 to 0 (resp. from 0 to 1) with rate $\lambda_1$ (resp. $\lambda_0$) i.e. $(X_t)$ solution of \eqref{eq:dec}. \\
Let us couple two paths $(X_t,I_t)$ and $(\tilde X_t,\tilde I_t)$ starting from $(x,i)$ 
and $(\tilde x,i)$ choosing the same jump times i.e. the same process ${(I_t)}_{t\geq 0}$. 
Then, the difference $D_t=X_t-\tilde X_t$ is solution of $\dot{D}_t=A_{I_t}D_t$.  

\begin{rem}
 If the initial discrete components are different, one has to wait for a 
random time $T$ following an exponential law of parameter $\lambda_0+\lambda_1$ and then, they can be chosen 
 equal for all $t\geq T$. 
\end{rem}

One can write $D_t=R_t U_t$ with $R_t>0$ and $U_t\in S^1$. Recall that 
\[
R_t=R_0 \exp\PAR{\int_0^t\cA(U_s,I_s)ds}.
\]
Let us define 
\[
\Psi_p(t)=\sup_{u\in S^1,i=0,1} \SBRA{\dE_{(u,i)} \PAR{\exp\PAR{p\int_0^t\cA(U_s,I_s)ds}}}^{1/p}.
\]
Thanks to Markov property, the function $\Psi_p$ satisfies $\Psi_p(t+s)\leq \Psi_p(t)\Psi_p(s)$. 
As a consequence, there exists $\chi_p$ such that 
\[
\chi_p=\lim_{t\to\infty} \frac{1}{t}\log \Psi_p(t).  
\]
Moreover, Jensen's inequality ensures that $\chi_p\geq \chi$. As a consequence, 
it is possible that $X_t\to 0$ a.s. with $\chi_p>0$ for some $p>0$.   

\begin{rem}
It is not easy to determine the sign of the Lyapunov exponent $\chi_p$.   
\end{rem}

\begin{lem}
When $\lambda_0$, $\lambda_1$ and $p$ are sufficiently small, there is convergence 
in Wasserstein distance for $(D_t)$ after the coupling of the discrete components.
\end{lem}

\begin{proof}
Once again, we assume that the initial discrete components are equal. Denote by 
$N_t$ the number of jumps for ${(I_t)}_{t\geq 0}$ before time $t$. One has
$I_t=(-1)^{N_t}I_0$. 
\[
D_t=\exp\PAR{(t-T_{N_t})A_{T_{N_t}}}
\exp\PAR{(T_{N_t}-T_{N_t-1})A_{I_{N_t-1}}}\cdots
\exp\PAR{(T_2-T_1) A_{I_{T_1}}}
\exp\PAR{T_1 A_{I_0}}D_0.
\]
Using Remark~\ref{rem:expo}, one has 
\[
\NRM{D_t}\leq C^{N_t+1} e^{-\eta t} \NRM{D_0}.
\]
If the jump rates of $I$ are equal, ($\lambda_0=\lambda_1=\lambda)$, then 
${(N_t)}_{t\geq 0}$ is a simple Poisson process and 
\[
\SBRA{\dE\PAR{\NRM{D_t}^p}}^{1/p} 
\leq C \exp\PAR{-\PAR{\eta-\frac{\lambda(C^p-1)}{p}}t}\SBRA{\dE\PAR{\NRM{D_0}^p}}^{1/p}.
\]
This ensures that when $\lambda$ and $p$ are sufficiently small one can establish convergence 
in Wasserstein distance, after the coupling of the discrete components (see \cite{bgm}).

 If the jump rates are not equal, ${(N_t)}_{t\geq 0}$ is not a Poisson process. For a fixed $c\in \dR$, 
 let us define $G_i(t)=\dE_i(c^{N_t})$. Then, if $T$ is the first jump time,   
\begin{align*}
G_1(t)&=\dE_1\PAR{c^{N_t}\ind_\BRA{T\leq t}}+\dE_1\PAR{c^{N_t}\ind_\BRA{T>t}}\\
&=\int_0^t c G_0(t-s)\lambda_1 e^{-\lambda_1 s}ds+ e^{-\lambda_1 t}.
\end{align*}
As a consequence, thanks to an integration by parts,
\begin{align*}
 G_1'(t)&=cG_0(0)\lambda_1 e^{-\lambda_1 t}+c\int_0^t G_0'(t-s)\lambda_1e^{-\lambda_1 s}ds
 -\lambda_1 e^{-\lambda_1 t}\\
&=c\lambda_1e^{-\lambda_1 t}+ c\SBRA{-G_0(t-s)\lambda_1 e^{-\lambda_1 t}}_0^t 
-c\int_0^t G_0(t-s)\lambda_1^2e^{-\lambda_1s}ds -\lambda_1e^{-\lambda_1t }\\
&=-\lambda_1G_1(t) +\lambda_1c G_0(t).
\end{align*}
Notice that, if $\lambda_1=\lambda_0$, one recovers that $G_0(t)=G_1(t)=\exp(\lambda t (c-1))$.
In the general case, $G_0$ and $G_1$ are solutions of 
\[
y''+(\lambda_0+\lambda_1)y'+\lambda_0\lambda_1(1-c^2)y=0
\]
with respective initial conditions $G_i(0)=1$, $G_i'(0)=\lambda_i(c-1)$.

After solving this simple equation, it also comes that for $\lambda$ and $p$ small enough, we can establish convergence in 
Wasserstein distance.
\end{proof}

We are now interested in the invariant measure, if it exists, of our process $(X_t)_{t\geq0}$ solution of \eqref{eq:dec}. 
Let us call $T_0 = 0$, $T_1$, ..., $T_n$ the jumping times of the process $(X_t)$ and $\tau_0$, $\tau_1$, ..., $\tau_n$,  
the interjump times following an exponential law of parameters $\lambda_0$ and $\lambda_1$ alternatively. Let us define $Y_n = X_{T_{2n}}$. By simply solving the equation \ref{eq:dec} between each 
jumping times we obtain:
\[
 Y_n = e^{\tau_{n+1} A_1}e^{\tau_n A_0} Y_{n-1} - e^{\tau_{n+1}} b_0 + b_1.
\]
Applying Kesten's renewal theorem (\cite{Kesten1973}) to the discrete process $(Y_n)$ will give some informations about the tail of the invariant measure of $(X_t)$.

\begin{thm}
If there is no deterministic control that makes the deterministic process associated to $(Y_t$) explodes, it means that there 
exists an invariant measure $\mu$ for our process and the support of $\mu$ is necessarily bounded.

In the other hand, we assume that $\chi < 0$ and that there exists a deterministic control such as:
\begin{equation}\label{eq:expl}
 \exists k\in 2\dN+1 \text{ such as }\forall x \in \dS_{d-1}, \exists t_0,t_1,...,t_k > 0 \text{ such as } \|e^{t_kA_{I_k}}...e^{t_1A_{I_1}}e^{t_0A_{I_0}}x\| > \|x\|.
\end{equation}
We call $B = e^{\tau_kA_{I_k}}...e^{\tau_1A_{I_1}}e^{\tau_0A_{I_0}}$ and 
$B_0$, $B_1$, ..., $B_n$ $n$ random variables i.i.d following the law of $B$. We also assume that:
\begin{equation}\label{eq:ret}
 \underset{n\geq0}{\max} \text{ } \dP \left( \frac{B_n...B_0x}{\|B_n...B_0x\|} \in U \right) >0 \text{ for all } x\in \dS_{d-1} \text{ and any open } \emptyset \neq U \subset \dS_{d-1}.
\end{equation}
\begin{equation}\label{eq:dens}
 \dP(B_{n_0}...B_0 \in .)\geq \gamma_0 \ind_{B_c(\Gamma_0)}\lambda \text{ for some } \Gamma_0 \in GL_d(\dR), \text{ } n_0 \in \dN \text{ and } c,\gamma_0>0.
\end{equation}
where $\lambda$ is the Lebesgue measure on $\dR^{n^2}$.
Under these assumptions, there exists a unique invariant measure $\mu$ for the process $(Y_t)$ and it has an heavy tail:
\[
\exists p_1>0 \text{ such as } \forall 0\leq p<p_1 \quad \dE_{\mu}[\|Y\|^p] < +\infty \text{ and } \forall p>p_1 \quad \dE_{\mu}[\|Y\|^p] = +\infty.
\]

\end{thm}

\begin{proof}Let $x\in \dS_{d-1}$. There exist $t_0$, $t_1$,... ,$t_k>0$ such as: 
\[
\| e^{t_kA_{I_k}}...e^{t_1A_{I_1}}e^{t_0A_{I_0}}x\| > \|x\|.
\]
Consequently, by continuity, there exist $\epsilon >0$ and $\delta >0$ such as:
\[
 \forall s_i \in J_i=(t_i-\delta,t_i+\delta), \quad \|e^{s_kA_{I_k}}...e^{s_1A_1}e^{s_0A_0}\|\geq1+\epsilon.
\]
Let us write $B = e^{\tau_kA_{I_k}}...e^{\tau_1A_1}e^{\tau_0A_0}$, where $\tau_0$, $\tau_1$,... ,$\tau_k$ are $k+1$ random exponential variables. We have:
\begin{align*}
\dE\left[ \|B\|^p\right] & \geq \dE \left[ \ind_{\tau_k\in J_k}...\ind_{\tau_1\in J_1}\ind_{\tau_0\in J_0} \|B\|^p \right]\\
                         & \geq \dP(\tau_k\in J_k)...\dP(\tau_1\in J_1)\dP(\tau_0\in J_0)(1+\epsilon)^p
\end{align*}
As $\tau_0$, $\tau_1$,... ,$\tau_k$ are $k+1$ exponential random variables, $\dP(\tau_k\in J_k)...\dP(\tau_1\in J_1)\dP(\tau_0\in J_0) >0$. Consequently:
\[
\dE\left[ \|B\|^p\right] \underset{p\rightarrow +\infty}{\longrightarrow} +\infty,
\]
meaning that there exists $x_1>0$ such as $\dE\left[ \|B\|^{x_1}\right] = 1$. Consequently, if the process $(Y_t)$ does not explode, 
\cite{gerold1} ensures that under the invariant measure $\mu$:
\[
\dP(\|Y\|>r) \sim \frac{c}{r^{x_1}}.
\]
It implies that $\mu$ has finite moments of order $p<x_1$.
\end{proof}

\begin{rem}
In dimension 2, hypothesis (\ref{eq:expl}) can be verified using the criteria in \cite{deter} which gives the existence of an 
explosive control for the switched system. In higher dimension, no general result gives this information.
\end{rem}

\begin{rem}
Hypothesis (\ref{eq:ret}) and (\ref{eq:dens}), directly extracted from a version of Kesten's renewal theorem in \cite{gerold1}, 
tells us, basically, that the process $(Y_t)$ creates density according to the Lebesgue measure and its projection on the sphere $\dS_{d-1}$ 
is a recurrent irreducible process. These hypothesis are not easy to check in the general case.
\end{rem}

\begin{cor}
 If there exists a deterministic control that makes the system explode from one of the stable point $b_0$ or $b_1$, then, 
since $A_0$ and $A_1$ are Hurwitz, the invariant measure has an heavy tail.
\end{cor}

\begin{xpl}
 In dimension 3, it is possible to create an example satisfying Theorem 6.5 by picking the following matrices:
\[
A_0=\begin{pmatrix}
 -a & b & 0\\
 -\frac{1}{b} & -a & 0\\
 0 & 0 & -1
\end{pmatrix}
\quad\text{and}\quad
A_1=\begin{pmatrix}
 -1 & 0 & 0\\
 0 & -a & b\\
 0 & -\frac{1}{b} & -a \\
\end{pmatrix}.
\]
When we consider each system $\dot{X_t}=A_iX_t$ separately, one can check that, when $a<1$, the projection of $(X_t)$ on the sphere 
tends to a different ecuador. 
Consequently, using results in \cite{bakt}, it is possible to show that the process $(X_t)$ solution of \eqref{eq:dec} satisfies hypothesis (\ref{eq:ret}) and (\ref{eq:dens}).

\end{xpl}

\textbf{Acknowledgments.} I acknowledge financial support from the French ANR project ANR-12-JS01-0006-PIECE.

\bibliographystyle{plain}
\bibliography{switched}

{\footnotesize %
 \noindent Gabriel \textsc{Lagasquie},
 e-mail: \texttt{gabriel.lagasquie(AT)univ-tours.fr}

 \medskip

 \noindent\textsc{Laboratoire de Mathématiques et Physique Théorique (UMR
CNRS 6083), Fédération Denis Poisson (FR CNRS
2964), Université François-Rabelais, Parc de Grandmont,
37200 Tours, France.}

\end{document}